\documentclass[a4paper,12pt]{article}
\usepackage[top=2.5cm,bottom=2.5cm,left=2.5cm,right=2.5cm]{geometry}
\usepackage{cite, amsmath, amssymb,color}
\usepackage{graphicx}

\newtheorem{theo}{Theorem}
\newtheorem{theorem}{Theorem}
\newtheorem{lemma}[theorem]{Lemma}

\newtheorem{corollary}[theo]{Corollary}

\newenvironment{proof}{\par \noindent \textbf{Proof: }}{\QED \par \bigskip \par}
\newcommand{\QED}{\hfill$\square$}

\allowdisplaybreaks[4]

\begin{document}

\baselineskip=0.30in

\vspace*{30mm}

\begin{center}
{\Large \bf \boldmath On the Maximum Sigma Index of $k$-Cyclic Graphs}

\vspace{10mm}

{\large \bf  Akbar Ali$^{1,}\footnote{Corresponding author}$, Abeer M. Albalahi$^{1}$, Abdulaziz M. Alanazi$^{2}$, Akhlaq A. Bhatti$^{3}$, Amjad E. Hamza$^{1}$}

\vspace{6mm}

\baselineskip=0.23in

$^1${\it Department of Mathematics, Faculty of Science,\\ University of Ha\!'il, Ha\!'il, Saudi Arabia}\\
{\tt akbarali.maths@gmail.com, a.albalahi@uoh.edu.sa, aboaljod2@hotmail.com}\\[3mm]
$^2${\it School of Mathematics, University of Tabuk,\\ Tabuk 71491, Saudi Arabia}\\
{\tt am.alenezi@ut.edu.sa}\\[3mm]
$^3${\it Department of Sciences and Humanities,\\ National University of Computer and Emerging Sciences, Lahore, Pakistan}\\
{\tt akhlaq.ahmad@nu.edu.pk}

\makeatletter

\def\@makefnmark{}

\makeatother


\vspace{4mm}

\baselineskip=0.23in

{\bf Abstract}

\end{center}
\noindent
Let $G$ be a graph with edge set $E(G)$. Denote by $d_w$ the degree of a vertex $w$ of $G$. The sigma index of $G$ is defined as $\sum_{uv\in E(G)}(d_u-d_v)^2$. A connected graph of order $n$ and size $n+k-1$ is known as a connected $k$-cyclic graph. Abdo, Dimitrov, and Gutman [Discrete Appl. Math. 250 (2018) 57--64] characterized the graphs having the greatest sigma index over the family of all connected graphs of a fixed order. The primary goal of the present note is to determine graphs possessing the greatest sigma index from the class of all connected $k$-cyclic graphs of a fixed order.   \\[3mm]
\\[3mm]

\baselineskip=0.35in

\section{Introduction}

Only finite graphs are considered in this study. The graph-theory terms that are used in this study, without providing their definitions, may be found in \cite{Bondy08,Harary69,Chartrand16}.

Although, regular graphs are usually thought of as graphs with the same vertex degrees (this definition of regular graphs is used throughout this paper), but a more general perspective is: graphs with some common property across their structure. Since the beginning of graph theory, regular graphs have been a major topic of research.
On the other hand, the study of graphs possessing a characteristic opposite to regularity has also attracted a much attention from researchers over the last couple of decades; such a study may be regarded as irregularity in graphs \cite{Ali-book}. To the best of authors' knowledge, the first paper devoted thoroughly to the concept of irregularity was written by Bezhad and Chartrand \cite{Bezhad-67}

For a graph $G$, its irregularity measure (IM) is a non-negative graph invariant fulfilling the property: every component of $G$ is regular if and
only if $IM=0$. Irregularity measures play a significant role in chemistry \cite{Reti-18,Gutman-05} and in network theory
\cite{Criado,Estrada-10,Estrada-10b,Snijders-81}.
One of the best known irregularity measures is the one introduced by Albertson \cite{Albertson-97} (which is often referred to as  the Albertson index):
\[
A(G) = \sum_{vw\in E(G)} |d_v-d_w|
\]
where $E(G)$ represents the edge set of a graph $G$ and $d_u$ denotes the degree of a vertex $u$ in $G$. In order to overcome the  Albertson index's several drawbacks, Abdo et al. \cite{Abdo-14} devised an extended version of the  Albertson index and named it as the total irregularity index. For a graph $G$, its total irregularity index is defined as
\[
irr_t(G) = \sum_{\{v,w\}\subseteq  E(G)} |d_v-d_w|.
\]

The present paper is mainly concerned with the sigma index \cite{Furtula-15,Gutman-18}, that is another variant of the Albertson index. For a graph $G$, its sigma index is denoted by $\sigma(G)$ and is defined as
\begin{equation*}
\sigma(G)=\sum_{uv\in E(G)}(d_{u}-d_{v})^{2}.
\end{equation*}

The study of the sigma index was explicitly initiated by Gutman et al. \cite{Gutman-18} where they not only established some fundamental properties of the sigma index but also proved that the sigma index of every graph is an even integer and constructed graph classes for which this index attains every positive even integer. R\'eti \cite{Reti-19-AMC} compared the sigma index with a couple of well-known irregularity measures and reported several interesting properties of this index.
Some additional mathematical characteristics of the sigma index can be found in the recent paper \cite{Lin-21} where the general Albertson index was proposed and studied. The main motivation of the present paper comes from the article \cite{Abdo-18} where the graphs having the greatest sigma index were characterized over the family of all connected graphs of a fixed order. In this paper,  the graphs possessing the greatest sigma index over the class of all connected $k$-cyclic graphs of a fixed order are determined, where a connected $k$-cyclic graph is a connected graph of order $n$ and size $n+k-1$.

\section{Main Results}

For a vertex $v$ in a graph $G$, denote by $N(v)$ the the set of all those vertices of $G$ that are adjacent to $v$. The members of $N(v)$ are known as neighbors of $v$.  A vertex $u\in V(G)$ of degree zero (one) is known as an isolated vertex (pendent vertex, respectively). The star graph with $n$ vertices is denoted by $S_n$. A graph of size zero is known as an edgeless graph.

\begin{lemma}\label{lem-2S-new-more}
For any graph $G$ of order $n$ and size $m$, with $0 \le m \le n-1$, the following inequality holds
\begin{equation*}
\sigma(G)\leq m(m-1)^{2}
\end{equation*}
with equality if and only if $G$ is either an edgeless graph \emph{(}for $m=0$\emph{)} or it consists of the star $S_{m+1}$ together with $n-m-1$ isolated vertices \emph{(}for $m\ge1$\emph{)}.
\end{lemma}

\begin{proof}
The integer $n$ is chosen to be fixed. The result is proved by using the mathematical induction on $m$. For $m =0$, the graph $G$ is the edgeless graph and for $m=1$ the graph $G$ consists of the star $S_{2}$ together with $n-2$ isolated vertices; in both cases, it holds that $\sigma(G)= m(m-1)^{2}=0$ and thence the induction starts.
Next, assume that $m = k\ge 2$. Choose an edge $uv\in E(G)$ such that $d_u\ge d_v$ and
\begin{equation}\label{eq-new-aal}
d_u-d_v = \max\big\{ |d_w - d_{w'}| :~ w'w \in E(G) \big\}.
\end{equation}
Take $Z=N(u)\cap N(v)$, $X=N(u)\setminus (Z\cup \{v\})$, and $Y=N(v)\setminus (Z\cup \{u\})$. Then, one has
\begin{eqnarray}\label{Eq-SSS}
\sigma(G)-\sigma(G-uv) &=& (d_{u}- d_{v})^{2} + \sum_{x\in X}\Big[(d_{u}- d_{x})^{2} - (d_{u}-1- d_{x})^{2}\Big]  \nonumber\\[2mm]
&&+ \sum_{y\in Y}\Big[(d_{v}- d_{y})^2 - (d_{v}-1- d_{y})^2\Big] \nonumber\\[2mm]
&& + \sum_{z\in Z}\Big[(d_{u}- d_{z})^{2} - (d_{u}- 1-d_{z})^{2} + (d_{v}- d_{z})^{2} - (d_{v}-1 -d_{z})^{2}\Big] \nonumber\\[2mm]
&=& (d_{u}- d_{v})^{2} + \sum_{x\in X}\Big[2(d_{u}- d_{x}) - 1\Big]  + \sum_{y\in Y}\Big[2(d_{v}- d_{y}) - 1\Big] \nonumber\\[2mm]
&& + 2\sum_{z\in Z}\Big[(d_{u}- d_{z}) + (d_{v}- d_{z}) - 1 \Big],
\end{eqnarray}
where every degree notation represents the degree in $G$ (not in $G-uv$).
By making use of Equation \eqref{eq-new-aal} in Equation \eqref{Eq-SSS}, one gets
\begin{eqnarray}\label{Eq-SSS-p0o}
\sigma(G)-\sigma(G-uv) &\le& (d_{u}- d_{v})^{2} + \sum_{x\in X}\Big[2(d_{u}- d_{v}) - 1\Big]  + \sum_{y\in Y}\big[2(d_{v}- d_{v}) - 1\big] \nonumber\\[2mm]
&& + 2\sum_{z\in Z}\Big[2(d_{u}- d_{v}) - 1 \Big]\nonumber\\[2mm]
&=& (d_{u}- d_{v})^{2} + \big(|X|+|Y|+2|Z|\big)\big[2(d_{u}- d_{v}) - 1\big]  .
\end{eqnarray}
Since $|X| + |Y| + 2|Z|=d_u+d_v - 2$, from \eqref{Eq-SSS-p0o} it follows that
\begin{eqnarray}\label{Eq-SSS-p0o-9}
\sigma(G)-\sigma(G-uv) &\le& (d_{u}- d_{v})^{2} + \big(d_{u}+ d_{v}-2\big)\big[2(d_{u}- d_{v}) - 1\big].
\end{eqnarray}
Since $k\ge2$, one has $d_u\ge2$.
Note that the function $\phi$ defined by
\[
\phi(t_1,t_2)= (t_1- t_2)^{2} + \big(t_1+ t_2-2\big)\big[2(t_1- t_2) - 1\big],
\]
with $k\ge t_1\ge t_2\ge 1$ and $t_1\ge2$, is strictly increasing in $t_1$ and strictly decreasing in $t_2$. This implies that $\phi(t_1,t_2)\le \phi(k,1)$ with equality if and only if $t_1=k, t_2= 1,$ and thence  \eqref{Eq-SSS-p0o-9} gives
\begin{eqnarray}\label{Eq-SSS-p0o-91}
\sigma(G)-\sigma(G-uv) &\le& (k- 1)^{2} + (k-1)(2k- 3)
\end{eqnarray}
Because of the inductive hypothesis, \eqref{Eq-SSS-p0o-91} yields
\begin{eqnarray}\label{Eq-SSS-p0o-912}
\sigma(G) &\le& (k- 1)^{2} + (k-1)(2k- 3) + (k-1)(k-2)^2 = k(k-1)^2.
\end{eqnarray}
Take $E'\subseteq E(G)$ such that if $rs\in E'$ then either $r=u$ or $s=v$.
From \eqref{Eq-SSS-p0o}, \eqref{Eq-SSS-p0o-9}, \eqref{Eq-SSS-p0o-91}, and \eqref{Eq-SSS-p0o-912},
we observe that the equation $\sigma(G) = k(k-1)^2$ holds if and only if
\begin{description}
  \item[(i).] the equation $d_u-d_v=|d_r-d_s|$ is satisfied for every edge $rs\in E'$,
  \item[(ii).] the vertex $u$ has the degree $k$ and the vertex $v$ is pendent,
  \item[(iii).] the graph $G-uv$ consists of the star $S_{k}$ together with $n-k$ isolated vertices.
\end{description}
Therefore, the induction (and hence the proof) is completed.

\end{proof}

The maximum degree of a graph $G$ is denoted by $\Delta(G)$.

\begin{lemma}\label{lem-0.5}
If $G$ is a graph possessing the greatest sigma index over the family of all connected $k$-cyclic graphs of a fixed order $n$, then $\Delta(G)=n-1$.
\end{lemma}

\begin{proof}
For $n \leq 3$, the result trivially holds. In what follows, we assume that $n \geq 4$.
Contrarily, assume that $v\in V(G)$ such that $d_v= \Delta(G)\le n-2$. Then, $G$ contains vertices $v',u_{1}$ such that $v'v,v'u_{1}\in E(G)$ but $vu_{1}\not\in E(G)$. Take $N(v')\setminus N(v):=\{v,u_{1},u_{2},\ldots,u_{p}\}$ where $p\geq1$. Also, take $A:=N(v)\setminus (N(v')\cup \{v'\})$ and $B:=N(v)\cap N(v')$.
Construct a new graph $G'$ from $G$ by dropping the edges
$u_{1}v',u_{2}v',\ldots,u_{p}v'$ and inserting the edges $u_{1}v,u_{2}v,\ldots,u_{p}v$.
In the remaining proof, by the vertex degree $d_{t}$ we mean degree of the vertex $t$ in the graph $G$.
One has
\begin{eqnarray}\label{Eq.2}
\sigma(G')-\sigma(G)&=&\sum_{ a\in A}\Big[(d_{v}+p-d_{a})^2-(d_{v}-d_{a})^2\Big]  \nonumber\\[2mm]
&&+\sum_{b\in B}\Big[(d_{v}+p-d_{b})^2-(d_{v}-d_{b})^2\Big]  \nonumber \\[2mm]
&&+\sum_{b\in B}\Big[(d_{v'}-p-d_{b})^2-(d_{v'}-d_{b})^2\Big]  \nonumber \\[2mm]
&&+\sum_{i=1}^{p}\Big[(d_{v}+p-d_{u_{i}})^2-(d_{v'}-d_{u_{i}})^2\Big]   \nonumber \\[2mm]
&&+(d_{v}+2p-d_{v'})^2-(d_{v}-d_{v'})^2  \nonumber\\[2mm]
&=&\sum_{ a\in A}\Big[(d_{v}+p-d_{a})^2-(d_{v}-d_{a})^2\Big]  \nonumber\\[2mm]
&&+\sum_{b\in B}2p(d_v - d_{v'} +p)  \nonumber \\[2mm]
&&+\sum_{i=1}^{p}\Big[(d_{v}+p-d_{u_{i}})^2-(d_{v'}-d_{u_{i}})^2\Big]   \nonumber \\[2mm]
&&+4p(d_v - d_{v'} +p).
\end{eqnarray}
Since $d_v= \Delta(G)$, from \eqref{Eq.2} it follows that $\sigma(G')-\sigma(G)>0$, which is a contradiction to the maximality of $\sigma(G)$ as $G'$ is also a connected $k$-cyclic graph with $n$ vertices.

\end{proof}

In order to prove the main results, we also need the following known result that is similar to Lemma \ref{lem-2S-new-more}.

\begin{lemma}\label{lem-2S-new-more-n}
{\rm \cite{Xu-14}} For any graph $G$ of order $n$ and size $m$, with $1 \le m \le n-1$, the following inequality holds
\begin{equation*}
M_1(G)\leq m(m+1)
\end{equation*}
with equality if and only if $G$ consists of the star $S_{m+1}$ together with $n-m-1$ isolated vertices.
\end{lemma}

For $n\ge4$, let $H_{n,k}$ be the graph constructed from the star graph $S_n$ by inserting $k$ edge(s) between $u\in V(S_n)$ and $k$ other pendent vertices, where $u$ is a fixed pendent vertex of $S_n$ and $k\ge1$. Also, we take $S_n=H_{n,0}$, see \cite{Ali19}.

\begin{figure}[!ht]
 \centering
  \includegraphics[width=0.4\textwidth]{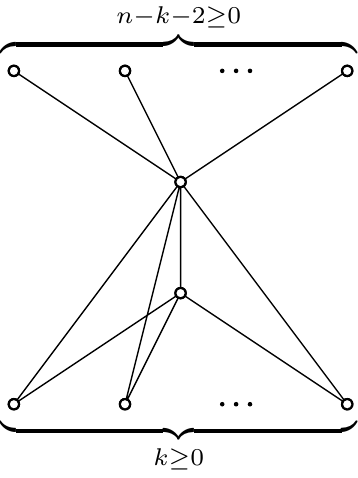}
   \caption{The $k$-cyclic graph $H_{n,k}$.}
    \label{Fig-1}
     \end{figure}

\begin{theorem}\label{thm-X2-more}
For $0\le k\leq n-2$ with $n\ge4$, if $G$ is connected $k$-cyclic graph with $n$ vertices, then
$$
\sigma(G)\le (n-1)(n-2)^2- 2k(2n-5)+ k^2(k-1),
$$
with equality if and only if $G=H_{n,k}$.
\end{theorem}

\begin{proof}
If the inequality $\Delta(G)<n-1$ holds, then by using the transformation (successively) used in the proof of Lemma \ref{lem-0.5}, one obtains a connected $k$-cyclic graph $G^{*}$ of order $n$ and maximum degree $n-1$ such that $\sigma(G)<\sigma(G^{*})$. Thereby, it enough to prove the result for the case when $\Delta(G)=n-1$. The result trivially holds for $k=0,1,$ because there is a unique connected graph of order $n$ and maximum degree $n-1$ in each such case. Thus, in what follows, assume that $k\ge2$ and $\Delta(G)=n-1$. Take a vertex $v \in V(G)$ such that $d_v=n-1$. Note that the size $|E(G-v)|$ of the graph $G-v$ is $k$, where $2\leq k\leq n-2= |V(G-v)|-1$. If $d'_w$ and $d_w$ are the degrees of a vertex $w\in V(G-v)$ in $G-v$ and $G$, respectively, then $d'_w =d_w-1$ for every $w\in V(G-v)$. If $E'=E(G-v)$ and $V'= V(G-v)$, then one has
\begin{eqnarray}\label{gtfrdx8-0}
\sigma(G) &=& \sum_{a\in V'}(d_v - d_a)^{2} + \sum_{uw\in E'}(d_u - d_w )^{2}\nonumber\\[2mm]
&=& \sum_{a\in V'}(n - d'_a -2)^{2} + \sum_{uw\in E'}(d'_u - d'_w )^{2}\nonumber\\[2mm]
&=& (n-1)(n-2)^2-4k(n-2)+ M_1(G-v) +  \sigma(G-v).
\end{eqnarray}
By using Lemmas \ref{lem-2S-new-more} and \ref{lem-2S-new-more-n} in \eqref{gtfrdx8-0}, we have
\begin{eqnarray}\label{gtfrdx8-0po}
\sigma(G) &\le& (n-1)(n-2)^2-4k(n-2)+ k(k+1) +  k(k-1)^2\label{hyyyygt}\\[2mm]
&=& (n-1)(n-2)^2- 2k(2n-5)+ k^2(k-1),\nonumber
\end{eqnarray}
where the equality sign in \eqref{hyyyygt} holds if and only if the graph $G-v$ consists of the star $S_{k+1}$ together with $n-1-k-1(=n-k-2)$ isolated vertices; that is, if and only if $G=H_{n,k}$.

\end{proof}

We observe that the graph $H_{n,k}$ exists whenever $0\le k\le n-2$. Thus, Theorem \ref{thm-X2-more} has the following direct consequence.

\begin{corollary}
If $0\le k\leq n-2$ and $n\ge4$, then $H_{n,k}$ uniquely possess the maximum sigma index over the collection of all connected $k$-cyclic graphs of a fixed order $n$.

\end{corollary}



\section*{Acknowledgment}
This research has been funded by Scientific Research Deanship, University of Ha\!'il -- Saudi Arabia through project
number RG-22\,002.

\end{document}